\documentclass{amsart}
\usepackage{amsmath,amssymb,amsfonts}
\usepackage[mathscr]{eucal}

\usepackage{enumerate}

\theoremstyle{plain}
\newtheorem{theorem}{Theorem}[section]

\newtheorem{prop}[theorem]{Proposition}
\newtheorem{lemma}[theorem]{Lemma}

\theoremstyle{definition}
\newtheorem{definition}[theorem]{Definition}

\input xy
\xyoption{all}

\newcommand{\Deltaop}{{\bf \Delta}^{op}}
\newcommand{\Thetanop}{\Theta_n^{op}}
\newcommand{\SSets}{\mathcal{SS}ets}

\newcommand{\Map}{\text{Map}}
\newcommand{\map}{\text{map}}
\newcommand{\Hom}{\text{Hom}}
\newcommand{\colim}{\text{colim}}

\newcommand{\Algt}{\mathcal Alg^\mathcal T}
\newcommand{\SSetst}{\mathcal{SS}ets^\mathcal T}

\newcommand{\SSetsd}{\mathcal{SS}ets^\mathcal D}
\newcommand{\alphau}{{\underline \alpha}}

\newcommand{\Sets}{\mathcal{S}ets}
\newcommand{\Tocat}{\mathcal T_{\mathcal {OC}at}}
\newcommand{\Algtocat}{\mathcal Alg^{\Tocat}}

\newcommand{\Thetaspdelta}{(\Theta_nSp)^{\Deltaop}}

\newcommand{\xu}{{\underline x}}
\newcommand{\Secat}{\mathcal Se \mathcal Cat}

\newcommand{\sesp}{\mathcal Se \mathcal Sp}
\newcommand{\sk}{\text{sk}}
\newcommand{\cosk}{\text{cosk}}
\newcommand{\css}{\mathcal{CSS}}
\newcommand{\hoequiv}{\text{hoequiv}}
\newcommand{\Ho}{\text{Ho}}

\newcommand{\ob}{\text{ob}}

\newcommand{\cof}{\text{cof}}

\newcommand{\Thetansp}{\Theta_nSp}

\begin{document}

\title[$(\infty,n)$-categories]{Comparison of models for $(\infty, n)$-categories, I}

\author[J.E. Bergner]{Julia E. Bergner}

\address{Department of Mathematics, University of California, Riverside, CA 92521}

\email{bergnerj@member.ams.org}

\author[C. Rezk]{Charles Rezk}

\address{Department of Mathematics, University of Illinois at Urbana-Champaign, Urbana, IL}

\email{rezk@math.uiuc.edu}

\date{\today}



\thanks{The first-named author was partially supported by NSF grants DMS-0805951 and DMS-1105766, and by a UCR Regents Fellowship.  The second-named author was partially supported by NSF grant DMS-1006054.}

\begin{abstract}
While many different models for $(\infty,1)$-categories are currently being used, it is known that they are Quillen equivalent to one another.  Several higher-order analogues of them are being developed as models for $(\infty, n)$-categories.  In this paper, we establish model structures for some naturally arising categories of objects which should be thought of as $(\infty,n)$-categories.  Furthermore, we establish Quillen equivalences between them.
\end{abstract}

\maketitle

\section{Introduction}

There has been much recent interest in homotopical notions of higher categories.  Given a positive integer $n$, an $n$-category has a notion of $i$-morphisms for all $1 \leq i \leq n$, and one can consider $\infty$-categories, in which there are $i$-morphisms for arbitrarily large $i$.  When such higher categories are considered as having strict associativity and unit laws on compositions at all levels, then their definitions are straightforward.  However, most examples of interest are better expressed as weak $n$-categories, where these laws are only required to hold up to isomorphism, and one needs to impose various coherence laws.  While there have been many proposed models for weak $n$-categories (often extending to models for weak $\infty$-categories), the problem of comparing these models has thus far been intractable.

However, in the world of homotopy theory, models for so-called $(\infty,1)$-categories, or $\infty$-categories with all $i$-morphisms invertible for $i>1$, have been far more manageable.  Several different approaches were taken, some originating from the idea of modeling homotopy theories, others with the intent of developing this kind of special case for higher category theory.  While these are by no means the only ones, four models for $(\infty,1)$-categories have been equipped with appropriate model structures: simplicial categories \cite{simpcat}, Segal categories \cite{hs}, \cite{pell}, quasi-categories \cite{joyal}, \cite{lurie}, and complete Segal spaces \cite{rezk}, and they have all been shown to be Quillen equivalent to one another \cite{survey}, \cite{thesis}, \cite{dugspiv}, \cite{hs}, \cite{joyal1}, \cite{jt}.

Simplicial categories, or categories enriched over simplicial sets, are probably the easiest to understand as $(\infty,1)$-categories, especially if we apply geometric realization and consider topological categories, or categories enriched over topological spaces.  Given any objects $x$ and $y$ in a topological category $\mathcal C$, the points of the mapping space $\Map_\mathcal C(x,y)$ can be regarded as 1-morphisms.  Paths between these points are 2-morphisms, but since paths can be reversed, these 2-morphisms are invertible up to homotopy.  Homotopies between these paths are 3-morphisms, and we can continue to take homotopies between homotopies to see that we have $n$-morphisms for arbitrarily large $n$, all of which are invertible up to homotopy.

Segal categories and quasi-categories are two different ways of thinking of weakened versions of simplicial categories, in which composition of mapping spaces is only defined up to homotopy.  Segal categories are bisimplicial sets with discrete space at level zero which satisfy a Segal condition, guaranteeing an up-to-homotopy composition.  Quasi-categories, on the other hand, are just simplicial sets, generally described in terms of a horn-filling condition which essentially gives the same kind of composition up to homotopy.

Like Segal categories, complete Segal spaces are bisimplicial sets satisfying the Segal condition, but instead of being discrete at level zero, they satisfy a ``completeness" condition that makes up for it: essentially, the spaces at level zero are weakly equivalent to the subspace of ``homotopy equivalences" sitting inside the space of morphisms.  The Quillen equivalence between the model structure for Segal categories and the model structure for complete Segal spaces tells us that this completeness condition exactly compensates for the discreteness of the level zero space in a Segal category.

While $(\infty,1)$-categories have been enormously useful in many ways, Lurie's recent proof of the cobordism hypothesis \cite{luriecob} has brought attention to the fact that they are not always good enough: for some purposes we need higher versions as well.  Thus, we can consider more general $(\infty,n)$-categories, or $\infty$-categories with $i$-morphisms invertible for $i>n$.  A few models for such objects have been proposed, namely the Segal $n$-categories of Hirschowitz-Simpson and Pelissier \cite{hs}, \cite{pell}, the $n$-fold complete Segal spaces of Barwick \cite{luriecob}, and the $\Theta_n$-spaces of the second-named author \cite{rezktheta}.  The latter model has the advantage that its model structure is cartesian closed.

In this paper, we seek to use the $\Theta_n$-space model to develop an $(\infty,n+1)$-analogue of simplicial categories.  Furthermore, we define a weakened version of it, which can be regarded as an $(\infty,n+1)$-version of Segal categories, but different from the Hirschowitz-Simpson model, and prove that the two are Quillen equivalent.  In fact, we have two different model structures for these higher Segal categories,

The model we propose for a higher-dimensional analogue of Segal categories is described in terms of functors $\Deltaop \rightarrow \Theta_nSp$, where $\Theta_nSp$ denotes the model category for $\Theta_n$-spaces, satisfying the Segal condition and a discreteness condition with respect to their being $\Deltaop$-diagrams.  We show that there exist two model structures, just as we have for ordinary Segal categories, which are Quillen equivalent to one another, and that they are in turn Quillen equivalent to the model category of categories enriched over $\Theta_nSp$.  This result generalizes the one establishing the Quillen equivalence between simplicial categories and Segal categories, i.e., the case where $n=1$ \cite{thesis}.  While only one of these model structures is necessary for this Quillen equivalence, the other one is the easier one to describe.  Furthermore, we anticipate, as in the $(\infty,1)$-case, that we will need the second one as we eventually seek to continue the zig-zag to establish the equivalence with $\Theta_{n+1}$-spaces.  These Quillen equivalences will be the subject of another paper.

Just as in the $(\infty, 1)$-category case, there are a number of preliminary results that need to be established.  We first show that we have appropriate model categories and Quillen equivalences when we restrict to Segal objects and the corresponding enriched categories which have a fixed set of objects.  To do so, we need to show that rigidification results of Badzioch on algebras over algebraic theories \cite{bad} continue to hold when we take these algebras in categories other than that of simplicial sets.

We also make use of our understanding of sets of generating cofibrations in a Reedy category, as well as the fact, established in a separate manuscript \cite{elegant}, that in this case the Reedy and injective model structures coincide.  By modifying these generating cofibrations appropriately, we are able to find a set of generating cofibrations for our more restrictive situation where the objects at level zero are discrete.  From there, we can find the more general model structures and prove the Quillen equivalence with the enriched categories much as we proved it in the earlier case.

\subsection{Work still to be done}
So far we have not extended the chain of Quillen equivalences to $\Theta_{n+1}Sp$, which would be the end goal, but there are a couple of possible approaches to doing so.  We expect to show that our model structure for Segal category objects is Quillen equivalent to the model category of complete Segal objects in $\Theta_nSp$, which is in turn Quillen equivalent to $\Theta_{n+1}Sp$.  This last step should use an inductive argument using the characterization of $\Theta_n$ as a wreath product of $n$ copies of ${\bf \Delta}$ \cite{berger} and be the first in a chain of Quillen equivalences between $\Theta_nSp$ and the model structure for Barwick's $n$-fold complete Segal spaces.  These results will be the subject of a future paper.

The results of this paper hold for more general cartesian presheaf categories other than $\Theta_nSp$.  However, the proofs require a good deal more subtlety, so these results will be given in a separate paper \cite{enrich}.  This problem has also been addressed by Simpson \cite{simpson}.

\subsection{Related work}
There are other models for $(\infty, n)$-categories as well as comparisons being established.  For example, Barwick has defined quasi-$n$-categories and compared them with $\Theta_n$-spaces; this model is also cartesian closed and therefore lends itself to defining a model via enrichment over it \cite{bar}.  In the case where $n=2$, Lurie has a model using Verity's complicial sets \cite{lurie2}, \cite{verity}.  Generalizing a result of To\"en \cite{toen}, Barwick and Schommer-Pries have developed a set of axioms which any model for $(\infty, n)$-categories must satisfy \cite{bsp}.  Ayala and Rozenblyum have also given a more geometric model for $(\infty,n)$-categories and have shown that it is Quillen equivalent to $\Theta_nSp$ \cite{ar}.

\subsection{Outline of the paper}
In Section 2 we review some basic material on model categories and simplicial objects, and in Section 3 we establish a model structure for categories enriched in $\Theta_nSp$.  In Sections 4 and 5, we generalize comparisons between Segal categories and simplicial categories in the fixed object set case to more general Segal category objects and enriched categories in $\Theta_nSp$.  Section 6 is devoted to establishing model structures for Segal category objects and in Section 7 we prove that they are Quillen equivalent to the model category of enriched categories.  In Section 8 we establish a technical result about fibrations in $\Thetansp$.

\section{Background}

Let ${\bf \Delta}$ denote the simplicial indexing category whose objects are the finite ordered sets $[n]=\{0<1< \cdots <n\}$ for $n \geq 0$.  Recall that a \emph{simplicial set} is a functor $\Deltaop \rightarrow \Sets$, where $\Sets$ denotes the category of sets.  Denote by $\SSets$ the category of simplicial sets.

A \emph{simplicial space} is a functor $\Deltaop \rightarrow \SSets$.  A simplicial set $X$ can be regarded as a simplicial space in two ways.  It can be considered a constant simplicial space with the simplicial set $X$ at each level, and in this case we will also denote
the constant simplicial set by $X$.  Alternatively, we can take the simplicial space,
which we denote $X^t$, for which $(X^t)_n$ is the discrete
simplicial set $X_n$.  The superscript $t$ is meant to suggest
that this simplicial space is the ``transpose" of the constant
simplicial space.

We recall some basics on model categories.  A \emph{model category} $\mathcal M$ is a category with three distinguished classes of morphisms: weak equivalences, fibrations, and cofibrations, satisfying five axioms \cite[3.3]{ds}.  Given a model category structure, one can define the \emph{homotopy category} $\Ho(\mathcal M)$, which is a localization of $\mathcal M$ with respect to the class of weak equivalences \cite[1.2.1]{hovey}.  An object $x$ in a model category $\mathcal M$ is \emph{fibrant} if the unique map $x \rightarrow \ast$ to the terminal object is a fibration.  Dually, an object $x$ in $\mathcal M$ is \emph{cofibrant} if the unique map $\varnothing \rightarrow x$ from the initial object is a cofibration.

Recall that an \emph{adjoint pair} of functors $F \colon \mathcal C \leftrightarrows \mathcal D
\colon G$ satisfies the property that, for any objects $X$ of
$\mathcal C$ and $Y$ of $\mathcal D$, there is a natural
isomorphism
\[ \varphi: \Hom_\mathcal D(FX, Y) \rightarrow \Hom_\mathcal C(X,
GY). \]  The functor $F$ is called the \emph{left adjoint} and $G$
the \emph{right adjoint} \cite[IV.1]{macl}.

\begin{definition} \cite[1.3.1]{hovey}
An adjoint pair of functors $F \colon \mathcal M \leftrightarrows
\mathcal N \colon G$ between model categories is a \emph{Quillen
pair} if $F$ preserves cofibrations and $G$ preserves fibrations.  The left adjoint $F$ is called a \emph{left Quillen functor}, and the right adjoint $G$ is called the \emph{right Quillen functor}.
\end{definition}

\begin{definition} \cite[1.3.12]{hovey}
A Quillen pair of model categories is a \emph{Quillen equivalence}
if for all cofibrant $X$ in $\mathcal M$ and fibrant $Y$ in
$\mathcal N$, a map $f \colon FX \rightarrow Y$ is a weak
equivalence in $\mathcal D$ if and only if the map $\varphi f
\colon X \rightarrow GY$ is a weak equivalence in $\mathcal M$.
\end{definition}

We will also need the notion of a simplicial model category
$\mathcal M$. For any objects $X$ and $Y$ in a simplicial category
$\mathcal M$, the \emph{function complex} is the simplicial set
$\Map(X,Y)$.

A \emph{simplicial model category} $\mathcal M$ is a model
category $\mathcal M$ that is also a simplicial category such that two axioms hold \cite[9.1.6]{hirsch}.

\begin{definition} \cite[17.3.1]{hirsch}
A \emph{homotopy function complex} $\Map^h(X,Y)$ in a simplicial
model category $\mathcal M$ is the simplicial set $\Map(\widetilde
X, \widehat Y)$ where $\widetilde X$ is a cofibrant replacement of
$X$ in $\mathcal M$ and $\widehat Y$ is a fibrant replacement for
$Y$.
\end{definition}

Several of the model category structures that we use are obtained
by localizing a given model category structure with respect to a
map or a set of maps.  Suppose that $P = \{f:A \rightarrow B\}$ is
a set of maps with respect to which we would like to localize a
model category $\mathcal M$.

\begin{definition} \label{local}
A $P$-\emph{local} object $W$ is a fibrant object of $\mathcal M$
such that for any $f:A \rightarrow B$ in $P$, the induced map on
homotopy function complexes
\[ f^*:\Map^h(B,W) \rightarrow \Map^h(A,W) \]
is a weak equivalence of simplicial sets.  A map $g:X \rightarrow
Y$ in $\mathcal M$ is a $P$-\emph{local equivalence} if for
every $P$-local object $W$, the induced map on homotopy function
complexes
\[ g^*: \Map^h(Y,W) \rightarrow \Map^h(X,W) \]
is a weak equivalence of simplicial sets.
\end{definition}

If $\mathcal M$ is a sufficiently nice model category, then one can obtain a new model structure with the same underlying category as $\mathcal M$ but with weak equivalences the $P$-local equivalences and fibrant objects the $P$-local objects \cite[4.1.1]{hirsch}.

Suppose that $\mathcal D$ is a small category and consider the category of functors $\mathcal D \rightarrow \SSets$, or $\mathcal D$-diagrams of spaces.  We would like to consider model category
structures on the category $\SSetsd$ of such diagrams. A natural choice for the weak equivalences in $\SSets^\mathcal D$
is the class of levelwise weak equivalences of simplicial sets.
Namely, given two $\mathcal D$-diagrams $X$ and $Y$, we define a
map $f:X \rightarrow Y$ to be a weak equivalence if and only if
for each object $d$ of $\mathcal D$, the map $X(d) \rightarrow
Y(d)$ is a weak equivalence of simplicial sets.

There is a model category structure $\SSetsd_f$ on the category of
$\mathcal D$-diagrams with these weak equivalences and in which
the fibrations are given by levelwise fibrations of simplicial
sets.  The cofibrations in $\SSetsd_f$ are then the maps of
simplicial spaces which have the left lifting property with
respect to the maps which are levelwise acyclic fibrations. This
model structure is often called the \emph{projective} model
category structure on $\mathcal D$-diagrams of spaces \cite[IX,
1.4]{gj}. Dually, there is a model category structure $\SSetsd_c$
in which the cofibrations are given by levelwise cofibrations of
simplicial sets, and this model structure is often called the
\emph{injective} model category structure \cite[VIII, 2.4]{gj}. In particular, we obtain these model structures for $\mathcal D=\Deltaop$, so that the category $\SSets^{\Deltaop}$ is just the category of simplicial spaces.

However, $\Deltaop$ is a Reedy category \cite[15.1.2]{hirsch}, and therefore we also have the Reedy model category structure on simplicial spaces
\cite{reedy}. In this structure, the weak equivalences are again
the levelwise weak equivalences of simplicial sets.  This model structure is cofibrantly generated, where the
generating cofibrations are the maps
\[ \partial \Delta[m] \times \Delta [n]^t \cup \Delta [m] \times
\partial \Delta [n]^t \rightarrow \Delta [m] \times \Delta [n]^t \]
for all $n,m \geq 0$, an the generating acyclic cofibrations are the maps
\[ V[m,k] \times \Delta [n]^t \cup \Delta [m] \times \partial \Delta
[n]^t \rightarrow \Delta [m] \times \Delta [n]^t \] for all $n
\geq 0$, $m \geq 1$, and $0 \leq k \leq m$ \cite[2.4]{rezk}.

However, for simplicial spaces, the Reedy model structure coincides with the injective model structure, as follows.

\begin{prop} \cite[15.8.7, 15.8.8]{hirsch} \label{inj}
A map $f:X \rightarrow Y$ of simplicial spaces is a cofibration in
the Reedy model category structure if and only if it is a
monomorphism. In particular, every simplicial space is Reedy
cofibrant.
\end{prop}

In light of this result, we denote the Reedy model structure
on simplicial spaces by $\SSets^{\Deltaop}_c$.  Both
$\SSets^{\Deltaop}_c$ and $\SSets^{\Deltaop}_f$ are simplicial
model categories. In each case, given two simplicial spaces $X$
and $Y$, we can define $\Map(X,Y)$ by
\[ \Map (X,Y)_n = \Hom_{\SSets^{\Deltaop}} (X \times \Delta [n],Y). \]

The projective model structure $\SSets^{\Deltaop}_f$ is also cofibrantly generated, and a set of generating cofibrations consists of the maps
\[ \partial \Delta [m] \times \Delta [n]^t \rightarrow \Delta [m]
\times \Delta [n]^t \] for all $m,n \geq 0$ \cite[IV.3.1]{gj}.

\section{Categories enriched in $\Theta_n$-spaces}

In this section, we begin with a summary of basic definitions and results for $\Theta_n$-spaces; a thorough treatment can be found at \cite{rezk} for $n=1$ and \cite{rezktheta} for the general case. We then establish a model for $(\infty, n+1)$-categories given by categories enriched in $\Theta_n$-spaces.  Since $\Theta_n$-spaces model $(\infty, n)$-categories, the model structure on these enriched categories is thus a higher-order version of the model structure on simplicial categories.

\begin{definition} \cite[4.1]{rezk}
A Reedy fibrant simplicial space $W$ is a \emph{Segal space} if for each $k \geq 2$ the Segal map
\[ \varphi_k: W_k \rightarrow \underbrace{W_1 \times_{W_0} \cdots \times_{W_0} W_1}_k \] is a weak equivalence of simplicial sets.
\end{definition}

\begin{theorem}  \cite[7.1]{rezk}
There is a cartesian closed model structure $\sesp$ on the category of simplicial spaces in which the fibrant objects are precisely the Segal spaces.
\end{theorem}

Because Segal spaces satisfy this Segal condition, we can regard them as being weakened versions of simplicial categories and apply appropriate terminology.  The \emph{objects} of a Segal space $W$ are the elements of the set $W_{0,0}$.  The \emph{mapping space} $\map_W(x,y)$ is given by the fiber of the map \[ (d_1, d_0): W_1 \rightarrow W_0 \times W_0 \] over $(x,y)$.  Since $W$ is Reedy fibrant, the fiber is in fact a homotopy fiber and therefore the mapping space is homotopy invariant.  Two maps $f,g \in \map_W(x,y)_0$ are \emph{homotopic} if they lie in the same component of the mapping space $\map_W(x,y)$.  The space of homotopy equivalences $W_{\hoequiv} \subseteq W_1$ is defined to be the union of all the components containing homotopy equivalences.  There is a (non-unique) way to compose mapping spaces, as given explicitly by the second-named author in \cite[\S 4]{rezk}.

The \emph{homotopy category} of $W$, denoted $\Ho(W)$, has objects the elements of the set $W_{0,0}$, and
\[ \Hom_{\Ho(W)}(x,y) = \pi_0 \map_W(x,y). \]  The image of a homotopy equivalence of $W$ in $\Ho(W)$ is an isomorphism.

We can consider maps between Segal spaces that are similar in structure to Dwyer-Kan equivalences of simplicial categories; we even give them the same name.

\begin{definition} \cite{rezk}
A map $f \colon W \rightarrow Z$ of Segal spaces is a \emph{Dwyer-Kan equivalence} if
\begin{enumerate}
\item for any objects $x$ and $y$ of $W$, the induced map $\map_W(x,y) \rightarrow \map_Z(fx,fy)$ is a weak equivalence of simplicial sets, and

\item the induced map $\Ho(W) \rightarrow \Ho(Z)$ is an equivalence of categories.
\end{enumerate}
\end{definition}

For a Segal space $W$, notice that the degeneracy map $s_0 \colon W_0 \rightarrow W_1$ factors through the space of homotopy
equivalences $W_{\hoequiv}$, since the image of $s_0$ consists of ``identity maps."

\begin{definition} \cite[\S 6]{rezk}
A Segal space $W$ is a \emph{complete Segal space} if the map $W_0 \rightarrow W_{\hoequiv}$ given above is a weak equivalence of simplicial sets.
\end{definition}

\begin{theorem} \cite[7.2]{rezk}
There is a cartesian closed model structure $\css$ on the category of simplicial spaces in which the fibrant objects are precisely the complete Segal spaces.
\end{theorem}

We now turn to $\Theta_n$-spaces as higher-order complete Segal spaces.  We begin by recalling the definition of the $\Theta$-construction.  Let $\mathcal C$ be a small category, and define $\Theta \mathcal C$ to be the category with objects $[m](c_1, \ldots, c_m)$ where $[m]$ is an object of ${\bf \Delta}$ and each $c_i$ is an object of $\mathcal C$.  A morphism
\[ [m](c_1, \ldots ,c_m) \rightarrow [q](d_1, \ldots, d_q) \] is given by $(\delta, \{f_{ij}\})$ where $\delta \colon [m] \rightarrow [q]$ in ${\bf \Delta}$ and $f_{ij} \colon c_i \rightarrow d_j$ are morphisms in $\mathcal C$ indexed by $1 \leq i \leq m$ and $1 \leq j \leq q$ where $\delta(i-1) < j \leq \delta (i)$ \cite[3.2]{rezktheta}.

Inductively, let $\Theta_0$ be the terminal category with a single object and no non-identity morphisms, and then define $\Theta_n=\Theta \Theta_n$.  Note that $\Theta_1={\bf \Delta}$.  The categories $\Theta_n$ have also been studied by Joyal and by Berger \cite{berger2}, \cite{berger}.

We can consider functors $\Thetanop \rightarrow \Sets$, and the most important example is the following.  For any object $[m](c_1, \ldots, c_m)$, let $\Theta[m](c_1, \ldots, c_m)$ be the analogue of $\Delta[m]$ in $\SSets$, i.e., the representable object for maps into $[m](c_1, \ldots, c_m)$.

Here, we consider functors $\Theta_n^{op} \rightarrow \SSets$.  Notice that any simplicial set can be regarded as a constant functor of this kind, and any functor $\Thetanop \rightarrow \Sets$, in particular the representable one given above, can be regarded as a levelwise discrete functor to $\SSets$.  Since, unlike in the case of simplicial spaces, the indexing diagrams in each direction are different, we can simply use the notation from the original category to denote such an object.  Since $\Thetanop$ is a Reedy category \cite{berger}, we have the Reedy model structure, as well as the projective and injective model structures, on the category $\SSets^{\Thetanop}$.  However, we prove in \cite{elegant} that the injective and Reedy model structures agree here, just as in the case of simplicial spaces.

Given $m \geq 2$ and $c_1, \ldots, c_m$ objects of $\Theta_n$, define the object
\[ G[m](c_1, \ldots, c_m)= \colim (\Theta[1](c_1) \leftarrow \Theta[0] \rightarrow \cdots \leftarrow \Theta[0] \rightarrow \Theta[1](c_m)). \]  There is an inclusion map
\[ se^{(c_1, \ldots, c_m)} \colon G[m](c_1, \ldots, c_n) \rightarrow \Theta[n](c_1, \ldots, c_m). \]  We define the set
\[ Se_{\Theta_n} = \{ se^{(c_1, \ldots, c_m)} \mid m \geq 2, c_1, \ldots c_m \in \ob(\Theta_n)\}. \]

However, being local with respect to these maps is not sufficient for our purposes, as it only gives an up-to-homotopy composition at level $n$.  Encoding lower levels of composition is achieved inductively, using the Segal object model structure on the category of functors $\Theta_n \rightarrow \SSets$.  This procedure is rather technical, and full details can be found in \cite[\S 8]{rezktheta}.  The main point is that, if the model structure on the category of functors $\Theta_{n-1} \rightarrow \SSets$ is obtained by localizing with respect to a set $\mathcal S$ of maps, we can make use of an intertwining functor $V \colon \Theta (\SSets^{\Theta_{n-1}^{op}}) \rightarrow \SSets^{\Thetanop}$ to translate the set $\mathcal S$ into a set $V[1](\mathcal S)$ of maps in $\SSets^{\Thetanop}$.  We will need to localize with respect to this set, in addition to those imposing the new Segal conditions for level $n$.
%
%In \cite[4.4]{rezktheta}, Rezk defines an intertwining functor
%\[ V \colon \Theta(\SSets^{\Theta_n^{op}}_c) \rightarrow \SSets^{\Thetanop}_c \] by
%\[ V[m](A_1, \ldots, A_m)([q](c_1, \ldots, c_q))= \coprod_{\delta \in \Hom_{\bf \Delta}([q], [m])} \prod_{i=1}^q \prod_{j=\delta(i-1)+1}^{\delta(i)} A_j(c_i) \] where the $A_j$ are objects of $\SSets^{\Theta_n^{op}}$ and the $c_i$ are objects of $\Theta_n$.  This functor can be used to ``upgrade" sets of maps in $\SSets^{\Theta_n^{op}}$ to sets of maps in $\SSets^{\Thetanop}$.  Given a map $f \colon A \rightarrow B$ in $\SSets^{\Theta_n^{op}}$, we obtain a map $V[1](f) \colon V[1](A) \rightarrow V[1](B)$.

Let $\mathcal S_1=Se_{\bf \Delta}$, and for $n \geq 2$, inductively define $\mathcal S_n=Se_{\Theta_n} \cup V[1](\mathcal S_{n-1})$.

\begin{theorem} \cite[8.5]{rezktheta}
Localizing the model structure $\SSets^{\Thetanop}_c$ with respect to $\mathcal S_n$ results in a cartesian model category whose fibrant objects are higher-order analogues of Segal spaces.
\end{theorem}

However, we need to incorporate higher-order completeness conditions as well.  To define the maps which respect to which we need to localize, we make use of an Quillen pair
\[ T_\# \colon \SSets^{\Deltaop}_c \rightarrow \SSets^{\Thetanop}_c \colon T^*\] to use known results for simplicial spaces \cite[4.1]{rezktheta}.   In particular, define
\[ Cpt_{\bf \Delta}= \{E \rightarrow \Delta[0]\} \] and, for $n \geq 2$,
\[ Cpt_{\Theta_n}=\{T_\# E \rightarrow T_\# \Delta[0]\}. \]  Let $\mathcal T_1=Se_{\Theta_1} \cup Cpt_{\Theta_1}$ and, for $n \geq 2$,
\[ \mathcal T_n=Se_{\Theta_n} \cup Cpt_{\Theta_n} \cup V[1](\mathcal T_{n-1}). \]

\begin{theorem} \cite[8.1]{rezktheta}
Localizing $\SSets^{\Thetanop}_c$ with respect to the set $\mathcal T_n$ gives a cartesian model category, denoted $\Theta_nSp$.
\end{theorem}

We refer to the fibrant objects of $\Theta_nSp$ simply as $\Theta_n$-\emph{spaces}.

As complete Segal spaces are known to be equivalent to simplicial categories, establishing them as models for $(\infty, 1)$-categories, $\Theta_{n+1}Sp$ should be Quillen equivalent to a model category whose objects are categories enriched in $\Theta_nSp$, further strengthening the view that they are indeed models for $(\infty, n+1)$-categories.

The existence of the appropriate model structure for enriched categories can be regarded as a special case of a result of Lurie \cite[A.3.2.4]{lurie}.

\begin{theorem} \label{vcat}
There is a cofibrantly generated model structure on the category $\Theta_nSp-Cat$ of small categories enriched in $\Theta_nSp$ in which the weak equivalences $f \colon \mathcal C \rightarrow \mathcal D$ are given by
\begin{itemize}
\item (W1) $\Hom_{\mathcal C}(x,y) \rightarrow \Hom_{\mathcal D}(fx,fy)$ is a weak equivalence in $\Theta_nSp$ for any objects $x,y$, and

\item (W2) $\pi_0 \mathcal C \rightarrow \pi_0 \mathcal D$ is an equivalence of categories, where $\pi_0 \mathcal C$ has the same objects as $\mathcal C$ and $\Hom_{\pi_0 \mathcal C} (x,y) = \Hom_{\Ho(\Theta_nSp)}(1, \mathcal C(x,y))$;
\end{itemize}
and the generating cofibrations are given by
\begin{itemize}
\item (I1) $\{ UA \rightarrow UB \}$ where $U \colon \Theta_nSp \rightarrow \Theta_nSp-Cat$ is the functor taking an object $A$ of $\Theta_nSp$ to the category with two objects $x$ and $y$, $\Hom_{UA}(x,y)=A$ and no other nonidentity morphisms, and $A \rightarrow B$ is a generating cofibration of $V$, and

\item (I2) $\varnothing \rightarrow \{x\}$, where $\{x\}$ denotes the category with one object and only the identity morphism.
\end{itemize}
\end{theorem}

Establishing that $\Theta_nSp-Cat$ is Quillen equivalent to $\Theta_{n+1}Sp$ should be achieved via a chain of Quillen equivalences, of which the ones shown in this paper are the beginning.

We will have need of the following generalizations of the definitions of Segal spaces.

\begin{definition}
A Reedy fibrant functor $W \colon \Deltaop \rightarrow \Theta_nSp$ is a $\Theta_nSp$-\emph{Segal space} if the Segal maps
\[ W_k \rightarrow \underbrace{W_1 \times_{W_0} \cdots \times_{W_0} W_1}_k \] are weak equivalences in $\Theta_nSp$ for all $k \geq 2$.
\end{definition}

\begin{theorem}
There is a cartesian closed model structure $\mathcal L_S (\Theta_nSp)^{\Deltaop}$ on the category of functors $\Deltaop \rightarrow \Theta_nSp$ in which the fibrant objects are precisely the Segal space objects in $\mathcal \Thetansp$.
\end{theorem}

\begin{proof}
To obtain the model structure, one can localize the Reedy model structure with respect to the analogues of the maps used to obtain the Segal space model structure.  To show that this model structure is cartesian, we follow the same line of argument as used by Rezk in \cite[\S 10]{rezk}.  First, we establish that any function object $W^X$ in $\Thetansp^{\Deltaop}$ is local, where $W^X$ is defined by
\[ (W^X)_{[q](c_1, \ldots, c_q), k} = \Hom(X \times \Theta[q](c_1, \ldots, c_q) \times \Delta[k], Y). \]

Regarding $\Delta[1]$ as a levelwise discrete object of $\Thetansp^{\Deltaop}$, consider the function object $W^{\Delta[1]}$ for any local object $W$.  Proving that $W^{\Delta[1]}$ is again local can be proved just as in Rezk's paper, using the notion of covering.  Then, for any $k \geq 2$, $W^{\Delta[k]}$ can be shown to be a retract of $W^{(\Delta[1])^k}$, establishing that $W^{\Delta[k]}$ is also local.  If $Y$ is any object of $\Thetansp$, regarded as a constant diagram in $\Thetansp^{\Deltaop}$, then $(W^{\Delta[k]})^Y= W^{\Delta[k] \times Y}$ is again local.  Since any object $X$ of $\Thetansp^{\Deltaop}$ can be written as a homotopy colimit of objects of the form $\Delta[k] \times Y$, any object of the form $W^X$ can be written as a homotopy limit of a objects of the form $W^{\Delta[k] \times Y}$, and therefore $W^X$ is local.

To complete the proof that this cartesian structure is compatible with the model structure, we can use the same argument as Rezk, using properties of adjoints.
\end{proof}

%\begin{definition}
%Let $E^{(k)}$ denote the free-standing $k$-isomorphism, regarded as a constant simplicial object in $\Theta_nSp$.  Let $\psi_k \colon E^{(k)} \rightarrow E^{(k-1)}$ be the collapse map.  Then a $\Theta_nSp$-Segal space is \emph{complete} if it is local with respect to the maps $\psi_k$ for $k \geq 1$.
%\end{definition}

%To get complete Segal space objects in $\Theta_nSp$, we need to assume that the indexing category $\mathcal C$ has a terminal object $c_0$.
%
%\begin{definition}
%Let $\Theta_nSp= \SSets^\mathcal C$ where $\mathcal C$ has a terminal object $c_0$.  Then a Segal space $W$ in $\Theta_nSp$, regarded as a functor $\Deltaop \times \mathcal C \rightarrow \SSets$, is \emph{complete} if
%\begin{enumerate}
%\item the map $W([0], c_0) \rightarrow W([0], c)$ is a weak equivalence of spaces for any object $c$ of $\mathcal C$, and
%
%\item each simplicial space $W(-, c)$ is a complete Segal space.
%\end{enumerate}
%If the model structure on $\Theta_nSp$ is a localization of the Reedy structure with respect to some set $S$ of maps, then we additionally require that for any $x,y \in W([0],c_0)_0$, the object of $\Theta_nSp$ given by
%\[ c \mapsto \holim \left( W([1],c) \rightarrow W([0],c)^2 \leftarrow \{x_y, y_c\} \right) \] is $S$-local, where $(x_c, y_c)$ is the image of $(x,y)$ under the map given in (1).
%\end{definition}
%
%Then we have a model structure $\mathcal L_S \Theta_nSp^{\Deltaop}$ in which the fibrant objects are the complete Segal spaces in $\Theta_nSp$.

\section{Fixed-object $\Thetansp$-Segal categories and their model structures}

In this section, we first recall basic definitions of Segal categories and generalize them to those of $\Thetansp$-Segal categories.  We then go on to establish model structures in the restricted case where all $\Thetansp$-Segal categories have the same set of objects which is preserved by all functions.

\begin{definition} \cite[\S 2]{hs}
A \emph{Segal precategory} is a simplicial space $X$ such that the simplicial set $X_0$ in degree zero is discrete, i.e. a constant simplicial set.
\end{definition}

Again, we can consider the Segal maps
\[\varphi_k: X_k \rightarrow \underbrace{X_1 \times_{X_0} \cdots \times_{X_0} X_1}_k \] for each $k \geq 2$. Since $X_0$ is discrete, the right-hand side is actually a homotopy limit.

\begin{definition} \cite[\S 2]{hs}
A \emph{Segal category} $X$ is a Segal precategory such that each Segal map $\varphi_k$ is a weak equivalence of simplicial sets for $k \geq 2$.
\end{definition}

There is a fibrant replacement functor $L$ taking a Segal precategory $X$ to a Segal category $LX$.  We can think of this functor as a ``localization," even though it is not actually obtained from localization of a different model structure \cite[\S 5]{thesis}.

Weak equivalences in this setting, again called \emph{Dwyer-Kan equivalences}, are the maps $f \colon X \rightarrow Y$
such that the induced map $\map_{LX}(x,y) \rightarrow
\map_{LY}(fx,fy)$ is a weak equivalence of simplicial sets for
any $x, y \in X_0$ and the map $\Ho(LX) \rightarrow \Ho(LY)$
is an equivalence of categories.

\begin{theorem} \cite[5.1, 7.1]{thesis}
There is a model structure $\Secat_c$ on the category of Segal precategories in which the fibrant objects are precisely the Reedy fibrant Segal categories.  The weak equivalences are the Dwyer-Kan equivalences.  There is also a model structure $\Secat_f$ with the same weak equivalences in which the fibrant objects are precisely the projective fibrant Segal categories.
\end{theorem}

\begin{theorem} \cite[7.5, 8.6]{thesis}
There is a chain of Quillen equivalences
\[ \mathcal {SC} \leftrightarrows \Secat_f \rightleftarrows \Secat_c \]
where $\mathcal {SC}$ denotes the model structure on the category of simplicial categories.
\end{theorem}

We would like to generalize these definitions and their corresponding model structures to $\Thetansp$-Segal categories; the goal of this paper is to prove the analogue of the previous theorem in this setting.

\begin{definition}
A $\Theta_nSp$-\emph{Segal precategory} is a functor $X \colon \Deltaop \rightarrow \Theta_nSp$ such that $X_0$ is a discrete object in $\Theta_nSp$, i.e., a constant $\Theta_n$-diagram of sets.  It is a $\Theta_nSp$-\emph{Segal category} if, additionally, the Segal maps
\[\varphi_k: X_k \rightarrow \underbrace{X_1 \times_{X_0} \cdots \times_{X_0} X_1}_k \] are weak equivalences in $\Theta_nSp$ for all $k \geq 2$.
\end{definition}

We denote by $\Theta_nSp^{\Deltaop}_{disc}$ the category of $\Theta_nSp$-Segal precategories.  Notice that if the Segal maps for $X$ are isomorphisms in $\Theta_nSp$, then $X$ is just a $\Theta_nSp$-category.

In the remainder of this section, we seek to define model structures on the category of functors $X \colon \Deltaop \rightarrow \Theta_nSp$ with the additional requirement that $X_0= \mathcal O$, the discrete object of $\Theta_nSp$ given by the a fixed set $\mathcal O$, and such that all maps between such functors are required to be the identity on this set.  We denote this category $\Theta_nSp^{\Deltaop}_\mathcal O$.

\begin{prop} \label{fixedf}
There is a model structure on $\Theta_nSp^{\Deltaop}_\mathcal O$ with levelwise weak equivalences and fibrations in $\Theta_nSp$, denoted by $\Theta_nSp^{\Deltaop}_{\mathcal O,f}$
\end{prop}

To prove this theorem, first notice that limits and colimits can be understood in this category just as they are in \cite[3.5,3.6]{simpmon}.  We then need sets of generating cofibrations and generating acyclic cofibrations for this proposed model structure.  The constructions here are generalizations of those for ordinary Segal categories \cite[\S 3]{simpmon}.

Just as we did in the case for simplicial sets, we begin by finding suitable sets of generating cofibrations and generating acyclic cofibrations for the projective model structure on the category $\Theta_nSp^{\Deltaop}$ of all functors $X \colon \Deltaop \rightarrow \Theta_nSp$.  By definition, a map $f \colon X \rightarrow Y$ in our proposed model structure is an acyclic fibration if and only if, for each $p \geq 0$, the map $f_p \colon X_p \rightarrow Y_p$ has the right lifting property with respect to every generating cofibration $A \rightarrow B$ in $\Theta_nSp$.  This condition is equivalent to the having a lift in the following diagram, for any $A \rightarrow B$ as above and $p \geq 0$:
\[ \xymatrix{A \times \Delta[p] \ar[r] \ar[d] & X \ar[d]^\simeq \\
B \times \Delta[p] \ar[r] \ar@{-->}[ur] & Y. } \]  Thus, we can regard the set of such maps
\[ A \times \Delta[p] \rightarrow B \times \Delta[p] \] as a suitable set of generating cofibrations for $\Theta_nSp$.
Similarly, $f$ is a fibration if and only if each $f_p$ has the right lifting property with respect to every generating acyclic cofibration $C \rightarrow D$ in $\Theta_nSp$.  It follows by arguments like the ones given above that a set of generating cofibrations consists of the maps
\[ C \times \Delta[p] \rightarrow D \times \Delta[p]. \]

Because the (constant) $\Theta_n$-space at level zero must be preserved, we need a distinct simplex of each dimension corresponding to each tuple of objects of $\mathcal O$.  Thus, for any $\xu=(x_0, \ldots, x_p) \in \mathcal O^{p+1}$, we define $\Delta[p]_{\xu}$ to be the $p$-simplex $\Delta[p]$, regarded as an object of $\Thetansp^{\Deltaop}_{disc}$, with $(\Delta[p]_{\xu})_0=\xu$.  Notice here that we assume that $\xu$ is ordered by the usual ordering on iterated face maps.  This object $\Delta[p]_\xu$ also contains all elements of $\mathcal O$ as 0-simplices.  It remains to find an appropriate means of assuring that each object involved in our generating (acyclic) cofibrations is in fact discrete in degree zero.

For any object $A$ in $\Theta_nSp$, $p\geq 0$, and $\xu \in \mathcal O^{p+1}$, define the object $A_{[p],\xu}$ to be the pushout of the diagram
\[ \xymatrix{A \times (\Delta[p]_{\xu})_0 \ar[r] \ar[d] & A \times \Delta[p]_\xu \ar[d] \\
(\Delta[p]_\xu)_0 \ar[r] & A_{[p],\xu}. } \]

Thus, we define sets
\[ I_{\mathcal O,f}= \{A_{[p], \xu} \rightarrow B_{[p],\xu} \mid p \geq 0, A \rightarrow B \text{ a generating cofibration in } \Theta_nSp \} \] and
\[ J_{\mathcal O,f}= \{C_{[p], \xu} \rightarrow D_{[p], \xu} \mid p \geq 0, C \rightarrow D \text{ a generating acyclic cofibration in } \Theta_nSp \}. \]

Given these generating sets, Proposition \ref{fixedf} can be proved just as in the simplicial case \cite[3.7]{simpmon}.

Now, we turn to the other model structure with levelwise weak equivalences, where we instead have levelwise cofibrations.  A useful fact is the following.

\begin{prop}
The Reedy and injective model structures on $\Theta_nSp^{\Deltaop}$ coincide.
\end{prop}

\begin{proof}
The fact that Reedy cofibrations are levelwise cofibrations in $\Thetansp$ follows from a general result about Reedy categories \cite[15.3.11]{hirsch}.  Therefore, it remains to prove that if $f \colon X \rightarrow Y$ in $\Thetansp^{\Deltaop}$ satisfies the condition that $f_n \colon X_n \rightarrow Y_n$ is a cofibration in $\Thetansp$, then $f$ is a Reedy cofibration.

We first need to understand what a ``codegeneracy" is in $\Theta_n$.  For simplicity, we look at $\Theta_2$.  Given an object $[k](c_1, \ldots ,c_k)$ in $\Theta_2$, there are two kinds of codegeneracies.  The first is given by a codegeneracy of a $c_i$, regarding $c_i$ as an object of ${\bf \Delta}$.  Using a ``pasting diagram" interpretation of $\Theta_2$, such a codegeneracy amounts to collapsing one of the 2-cells at horizontal position $i$.  Thus, when we take a simplicial presheaf on $\Theta_2$, the corresponding degeneracy gives a degenerate 2-cell in a position specified by the degeneracy map of the $c_i$ in $\Deltaop$.  We think of such degeneracies as ``vertical" degeneracies.

There is also a kind of ``horizontal" degeneracy, but we do not want to allow all such.  Given an object $[k](c_1, \ldots ,c_k)$, a horizontal degeneracy would be given by a codegeneracy of $[k]$ in $\Delta$.  But, if we took the $i$th codegeneracy of $[k]$, where $c_i>0$, then we would, in effect, we collapsing multiple cells.  Thus, we only want to consider such codegeneracies when $c_i=0$, i.e., the case where there are no 2-cells in position $i$.

In either case, however, a degeneracy is given by a degeneracy in $\Deltaop$, and therefore our result about degeneracies in $\Deltaop$ continues to hold in $\Theta_2^{op}$.  This argument can be rephrased as an inductive one, so that it is in fact true for all $\Thetanop$.

Now, we establish an analogue of \cite[15.8.6]{hirsch} in this situation, namely, that, for every $m \geq 0$, the latching object $L_mX$ is isomorphic to the subobject of $X_m$ consisting of higher-order simplices, i.e., objects of $\Hom(\Theta[m](c_1, \ldots, c_m), X)$, which are in the image of a degeneracy operator.  However, this fact follows from \cite[15.8.4]{hirsch} and the existence of a map from $(L_mX)_{[k](c_1, \ldots, c_k)}$ to the degenerate elements of $X_{\ast, [k](c_1, \ldots, c_k)}$.

Using this above description of codegeneracies in $\Theta_n$, we have the analogue of \cite[15.8.5]{hirsch}, that for any object $W$ of $\Thetansp$, if $k \geq 0$, $\sigma \in W_{[k](c_1, \ldots, c_k)}$ is nondegenerate if and only if no two degeneracies of $\sigma$ are equal.  Therefore, it follows that the intersection of $X_m$ and $L_mY$ in $Y_m$ is precisely the object $L_mX$.  Therefore, the latching map $X_m \amalg_{L_mX} L_mY \rightarrow Y_n$ is an monomorphism in $\Thetansp$, which is precisely the requirement for $f$ to be a Reedy cofibration.
\end{proof}

Thus, we can use the Reedy structure to understand precise sets of generating cofibrations and generating acyclic cofibrations, but we also know that cofibrations are precisely the monomorphisms and in particular that all objects are cofibrant.

\begin{prop}
There is a model structure on $\Theta_nSp^{\Deltaop}_\mathcal O$ with levelwise weak equivalences and cofibrations in $\Theta_nSp$, denoted by $\Theta_nSp^{\Deltaop}_{\mathcal O,c}$
\end{prop}

To define sets $I_{c, \mathcal O}$ and $J_{c, \mathcal O}$ which will be our candidates for generating cofibrations and generating acyclic cofibrations, respectively, we first recall the generating cofibrations and acyclic cofibrations in the Reedy model structure.  The generating cofibrations are the
maps
\[ A \times \Delta [p] \cup B \times \partial
\Delta [p] \rightarrow B \times \Delta [p] \] for all
$p \geq 0$ and $A \rightarrow B$ generating cofibrations in $\Theta_nSp$, and similarly the generating acyclic cofibrations
are the maps
\[ C \times \Delta [p] \cup C \times \partial \Delta
[p] \rightarrow D \times \Delta [p] \] for all $p
\geq 0$ and $C \rightarrow D$ generating acyclic cofibrations in $\Theta_nSp$ \cite[15.3]{hirsch}.

To modify these maps, we begin by considering the category $\Theta_nSp^{\Deltaop}_{disc}$ of all Segal precategory objects in $\Theta_nSp$ and the inclusion functor
$\Theta_nSp^{\Deltaop}_{disc} \rightarrow \Theta_nSp^{\Deltaop}$. This functor has a left
adjoint which we call the reduction functor. Given an object $X$ of $\Theta_nSp^{\Deltaop}$, we denote its reduction by $(X)_r$. Reducing $X$ essentially amounts to collapsing the space $X_0$ to its set of components and making the appropriate changes to degeneracies in
higher degrees.  So, we start by reducing the objects defining the
Reedy generating cofibrations and generating acyclic cofibrations
to obtain maps of the form
\[ (A \times \Delta [p] \cup B \times \partial \Delta [p])_r \rightarrow (B \times \Delta [p])_r \]
and
\[ (C \times \Delta [p] \cup D \times \partial \Delta [p])_r \rightarrow (D \times \Delta [p])_r \] Then,
in order to have our maps fix the object set $\mathcal O$, we define a separate such
map for each choice of vertices $\xu$ in degree zero and adding in
the remaining points of $\mathcal O$ if necessary.  As above, we
use $\Delta [p]_\xu$ to denote the object $\Delta [p]$ with
the $(p+1)$-tuple $\xu$ of vertices. We then define sets
\[ I_{\mathcal O, c} = \{(A \times \Delta [p]_\xu \cup B \times \partial \Delta [p]_\xu)_r \rightarrow (B \times
\Delta [p]_\xu)_r \} \] for all $p \geq 1$ and $A \rightarrow B$, and
\[ J_{\mathcal O, c} = \{(C \times \Delta [p]_\xu \cup D \times \partial \Delta [p]_\xu)_r \rightarrow (D \times \Delta
[p]_\xu)_r\} \] for all $p \geq 1$ and $C \rightarrow D$, where the notation $(-)_\xu$ indicates the specified vertices.

Then, the proof that we do in fact get a model structure can be proved just as in \cite[3.9]{simpmon}.

However, these two model structures are not enough.  We need to localize them so that their fibrant objects are Segal category objects, following \cite{rezk}.  Fortunately, this process can be done just as in the $n=1$ case.  Define a map $\alpha^i:[1] \rightarrow [p]$ in ${\bf
\Delta}$ such that $0 \mapsto i$ and $1 \mapsto i+1$ for each $0
\leq i \leq p-1$. Then for each $p$ defines the object
\[ G(p)= \bigcup_{i=0}^{p-1} \alpha^i \Delta [1] \] and the inclusion map $\varphi^p:
G(p) \rightarrow \Delta [p]$.  To obtain the Segal model structure from the Reedy model structure on the category of functors $\Deltaop \rightarrow \Theta_nSp$, the localization is with
respect to the coproduct of inclusion maps
\[ \varphi = \coprod_{p \geq 0} (G(p) \rightarrow \Delta [p]). \]

However, in our case, the objects $G(p)$ and $\Delta [p]$ do  not preserve the object set.  As before, we can replace $\Delta
[p]$ with the objects $\Delta [p]_\xu$, where $\xu =(x_0,
\ldots ,x_p)$ and define
\[ G(p)_\xu = \bigcup_{i=0}^{p-1} \alpha^i \Delta
[1]_{x_i, x_{i+1}}. \] Now, we need to take coproducts not only
over all values of $p$, but also over all $p$-tuples of vertices.  Here, we can regard these objects as giving a diagram of constant $\Theta_n$-spaces.

Thus, we localize with respect to the set of maps
\[ \{G[p]_\xu \rightarrow \Delta[p]_\xu \mid p \geq 0, \xu \in \mathcal O^{p+1} \}. \]  Applying this localization to the model structure $\Theta_nSp^{\Deltaop}_{\mathcal O, f}$ gives a model structure which we denote $\mathcal L(\Thetansp)^{\Deltaop}_{\mathcal O,f}$, and similarly from $\Theta_nSp^{\Deltaop}_{\mathcal O,c}$ we obtain the localized model structure $\mathcal L (\Theta_nSp)^{\Deltaop}_{\mathcal O,c}$.

\section{Rigidification of algebras over algebraic theories}

In this section we generalize work of Badzioch \cite{bad} and the first-named author \cite{multisort} concerning rigidification of simplicial algebras over algebraic theories.  These results, which give us a convenient framework for understanding fixed-object simplicial categories, were used to establish the Quillen equivalence between the model structures for simplicial categories and Segal categories.  To apply these results to our higher-categorical situation, we need similar results to hold when we take functors into more general categories.

We begin with a review of algebraic theories and simplicial algebras over them.

\begin{definition} \cite{multisort}
Given a set $S$, an $S$-\emph{sorted algebraic theory} (or \emph{multi-sorted theory}) $\mathcal T$ is a small category with objects $T_{\alphau^n}$ where $\alphau^n = \langle \alpha_1, \ldots
,\alpha_n \rangle$ for $\alpha_i \in S$ and $n \geq 0$ varying, and such that each $T_{\alphau^n}$ is equipped with an isomorphism
\[ T_{\alphau^n} \cong \prod_{i=1}^n T_{\alpha_i}. \]
For a particular $\alphau^n$, the entries $\alpha_i$ can repeat, but they are not ordered.  In other words, $\alphau^n$ is a an $n$-element subset with multiplicities.  There exists a terminal
object $T_0$ corresponding to the empty subset of $S$.
\end{definition}

\begin{definition} \label{talg}
Given an $S$-sorted theory $\mathcal T$, a \emph{(strict simplicial)} $\mathcal T$-\emph{algebra in} $\Theta_nSp$ is a product-preserving functor $A:\mathcal T \rightarrow \Theta_nSp$.  In other words, the canonical map
\[ A(T_{\alphau^n}) \rightarrow \prod_{i=1}^n A(T_{\alpha_i}), \]
induced by the projections $T_{\alphau^n} \rightarrow T_{\alpha_i}$ for all $1 \leq i \leq n$, is an isomorphism in $\Theta_nSp$.
\end{definition}

We denote the category of strict $\mathcal T$-algebras in $\Theta_nSp$ by $\Algt_{\Theta_n}$.

\begin{definition}
Given an $S$-sorted theory $\mathcal T$, a \emph{homotopy} $\mathcal T$-\emph{algebra in} $\Theta_nSp$ is a functor $X:\mathcal T \rightarrow \Theta_nSp$ which preserves products up to homotopy, i.e.,
for all $\alpha \in S^n$, the canonical map
\[ X(T_{\alphau^n}) \rightarrow \prod_{i=1}^n X(T_{\alpha_i}) \]
induced by the projection maps $T_{\alphau^n} \rightarrow T_{\alpha_i}$ for each $1 \leq i \leq n$ is a weak equivalence in $\Theta_nSp$.
\end{definition}

Given an $S$-sorted theory $\mathcal T$ and $\alpha \in S$, there is an evaluation functor
\[ U_\alpha \colon \Algt_{\Theta_n} \rightarrow \Theta_nSp \] given by
\[ U_\alpha(A)=A(T_\alpha). \]
Define a weak equivalence in the category $\Algt_{\Theta_n}$ to be a map $f \colon A \rightarrow B$ such that $U_\alpha(f) \colon U_\alpha(A) \rightarrow U_\alpha(B)$ is a weak equivalence in $\Theta_nSp$ for all $\alpha \in S$.  Similarly, define a fibration of $\mathcal T$-algebras to be a map $f$ such that $U_{\alpha}(f)$ is a fibration in $\mathcal  M$ for all $\alpha$. Then define a cofibration to be a map with the left lifting property with respect to the maps which are fibrations and weak equivalences.

The following theorem is a generalization of a result by Quillen \cite[II.4]{quillen}.

\begin{prop}
There is a model structure on the category $\Algt_{\Theta_n}$ with weak equivalences and fibrations given by evaluation functors $U_\alpha$ for all $\alpha \in S$.
\end{prop}

\begin{proof}
The proof follows just as it does for algebras in $\SSets$ \cite[4.7]{multisort}.
\end{proof}

Let $\Theta_nSp^\mathcal T_f$ denote the category of functors $\mathcal T \rightarrow \Theta_nSp$ with model structure given by levelwise weak equivalences and fibrations.  Similarly, let $\Theta_nSp^\mathcal T_c$ denote the same category with model structure given by levelwise weak equivalences and cofibrations.  Since the objects of $\Theta_nSp$ are simplicial presheaves, in particular presheaves of sets, we can regard the set of maps
\[ P=\{p_{\alphau^n} \colon \coprod_{i=1}^n \Hom_\mathcal T (T_{\alpha_i},-) \rightarrow \Hom_\mathcal T(T_{\alphau^n},-) \} \] as defining a set of maps in $\Theta_nSp$ given by constant diagrams.  Then, we have model structures $\mathcal L (\Theta_nSp)^\mathcal T_f$ and $\mathcal L(\Thetansp)^\mathcal T_c$ given by localizing the model structures $\Theta_nSp^\mathcal T_f$ and $\Theta_nSp^\mathcal T_c$ with respect to this set of maps.  The following proposition generalizes \cite[4.9]{multisort}.

\begin{prop}
There is a model category structure $\mathcal L \Theta_nSp^\mathcal T$ on the category $\Theta_nSp^\mathcal T$ with weak equivalences the $P$-local equivalences, cofibrations as in $\SSetst_f$, and fibrations the maps which have the right lifting property with respect to the maps which are cofibrations and weak equivalences.
\end{prop}

Here, we use a slight modification of this theorem as follows.  We define a model structure analogous to $\mathcal L \Theta_nSp^\mathcal T$ but on the category of functors $\mathcal T \rightarrow \Theta_nSp$ which send $T_0$ to $\Delta [0]$, as in \cite[3.11]{simpmon}.

\begin{prop}
Consider the category of functors $\mathcal T \rightarrow \Theta_nSp$ such that the image of $T_0$ is $\Delta [0]$.  There is a model category structure on $\mathcal L(\Thetansp)^\mathcal T_*$ in which the in which the fibrant objects are homotopy $\mathcal T$-algebras in $\Theta_nSp$.
\end{prop}

The main theorem of this section is the following, and its proof follows just as in the case of $\SSets$.

\begin{theorem}
There is a Quillen equivalence of model categories
\[ \xymatrix@1{L:\mathcal L (\Theta_nSp)^\mathcal T_{*,f} \ar@<.5ex>[r] & \Algt_{\Theta_n}:N. \ar@<.5ex>[l]} \]
\end{theorem}

We now look at the algebraic theory that is of use here, namely the theory $\Tocat$ of categories with fixed object set $\mathcal O$. Consider the category $\mathcal {OC}at$ whose objects are the small categories with a fixed object set $\mathcal O$ and whose morphisms are the functors which are the identity on the objects. There is a theory $\mathcal T_{\mathcal {OC}at}$ associated to this category. Given an element $(\alpha, \beta) \in \mathcal O \times \mathcal O$, consider the directed graph with vertices the elements of $\mathcal O$ and with a single edge starting at $\alpha$ and ending at $\beta$.  The objects of $\Tocat$ are isomorphism classes of categories which are freely generated by coproducts of such directed graphs   In other words, this theory is $(\mathcal O \times \mathcal O)$-sorted.

A product-preserving functor $\Tocat \rightarrow \Sets$ is essentially a category with object set $\mathcal O$.  In the comparison between simplicial categories and Segal categories with a fixed object set, we use simplicial algebras $\Tocat \rightarrow \SSets$, which correspond to simplicial categories, or categories enriched over simplicial sets, with fixed object set $\mathcal O$.  Here, we regard strictly product-preserving functors $\Tocat \rightarrow \Theta_nSp$ as categories enriched over $\Theta_nSp$ with object set $\mathcal O$.

When $\Theta_nSp$ is additionally a cofibrantly generated model category of simplicial presheaves, then we can consider the model structure $\Algtocat_{\Theta_n}$ and the related model structure for homotopy algebras, $\mathcal L(\Thetansp)^{\Tocat}$.  The homotopy algebras can be regarded as a weaker version of categories enriched over $\Theta_nSp$, yet not as weak as the Segal category objects that we considered in the previous section; our goal is to show they are all equivalent nonetheless.

We first note the easiest such equivalence.

\begin{prop}
The identity functor gives a Quillen equivalence
\[ \xymatrix@1{\mathcal L (\Theta_nSp)^{\Deltaop}_{\mathcal O,f} \ar@<.5ex>[r] & \mathcal L (\Theta_nSp)^{\Deltaop}_{\mathcal O,c}. \ar@<.5ex>[l]} \]
\end{prop}

\begin{proof}
The proof follows since weak equivalences are the same in both model structures and all the cofibrations in $\mathcal L (\Theta_nSp)^{\Deltaop}_{\mathcal O,f}$ are cofibrations in $\mathcal L (\Theta_nSp)^{\Deltaop}_{\mathcal O,c}$.
\end{proof}

The following proof is more difficult to establish, but in fact the argument is identical to case of $\SSets$ \cite[\S 4, \S 5]{simpmon}.

\begin{theorem}
There is a Quillen equivalence of model categories
\[ \xymatrix@1{\mathcal L(\Thetansp)^{\Tocat}_{\mathcal O,f} \ar@<.5ex>[r] & \mathcal L(\Theta_nSp)^{\Deltaop}_{\mathcal O,f}. \ar@<.5ex>[l]} \]
\end{theorem}

\section{Two model structures for Segal category objects}

We begin by defining sets of maps which will be our generating cofibrations in our two model structures.  However, here we no longer require object sets to remain fixed.

Thus, we begin with the generating cofibrations for the Reedy model structure on $\Theta_nSp^{\Deltaop}_c$, which are given by
\[ A \times \Delta[p] \cup B \times \partial \Delta [p] \rightarrow B \times \Delta[p], \] where $A \rightarrow B$ ranges over all generating cofibrations in $\Theta_nSp$ and $p \geq 0$.  Since the localization does not change the cofibrations, we can use the Reedy generating cofibrations as a generating set for $\Theta_nSp$.  Recall that a map $X \rightarrow Y$ is an acyclic fibration in $\SSets^{\Theta_n}$ if, for any object $[q](c_1, \ldots, c_q)$, the map $X_{(c_1, \ldots, c_q)} \rightarrow P_{(c_1, \ldots,c_q)}$ is an acyclic fibration of simplicial sets, where $P_{(c_1, \ldots, c_q)}$ is the pullback in the diagram
\[ \xymatrix{P_{(c_1, \ldots, c_q)} \ar[r] \ar[d] & Y_{(c_1, \ldots, c_q)} \ar[d] \\
M_{(c_1, \ldots, c_q)} X \ar[r] & M_{(c_1, \ldots, c_q)}Y. } \]  Here $M_{(c_1, \ldots, c_q)}X$ denotes the matching object for $X$ at $[q](c_1, \ldots, c_q)$ and analogously for $Y$ \cite[15.2.5]{hirsch}.

The map $X_{(c_1, \ldots, c_q)} \rightarrow P_{(c_1, \ldots,c_q)}$ is an acyclic fibration of simplicial sets precisely when it has the left lifting property with respect to the generating cofibrations for the standard model structure on $\SSets$, i.e., with respect to the maps $\partial \Delta[m] \rightarrow \Delta[m]$ for all $m \geq 0$.  Now, notice that
\[ X_{(c_1, \ldots, c_q)} = \Map(\Theta[q](c_1, \ldots, c_q), X) \] and
\[ M_{(c_1, \ldots, c_q)}X = \Map(\partial \Theta[q](c_1, \ldots, c_q), X) \] where $\Theta[q](c_1, \ldots, c_q)$ is the analogue of $\Delta[q]$ in $\SSets$, i.e., the representable object for maps into $[q](c_1, \ldots, c_q)$, and $\partial \Theta[q](c_1, \ldots, c_q)$ is the analogue of $\partial \Delta[q]$.
Thus, we get that
\[ P_{(c_1, \ldots, c_p)} = \Map(\Theta[q](c_1, \ldots, c_q),Y) \times_{\Map(\partial \Theta[q](c_1, \ldots, c_q), Y)} \Map(\partial \Theta[q](c_1, \ldots, c_q),X). \]
Putting all this information together, we see that $X \rightarrow Y$ is an acyclic fibration in $\Thetansp$ precisely when it has the right lifting property with respect to all maps
\[ \partial \Delta[m] \times \Theta[q](c_1, \ldots c_q) \cup \Delta[m] \times \partial \Theta[q](c_1, \ldots, c_q) \rightarrow \Delta[m] \times \Theta[q](c_1, \ldots, c_q). \]

Thus, returning to the setting of $\Thetaspdelta_{disc}$, we have a preliminary set of possible generating cofibrations given by

\SMALL

\[ \begin{aligned}
\left( (\partial \Delta[m] \times \Theta[q](c_1, \ldots, c_q) \cup \Delta[m] \times \partial \Theta[q](c_1, \ldots, c_q)) \times \Delta[p] \cup (\Delta[m] \times \Theta[q](c_1, \ldots, c_q)) \times \partial \Delta[p] \right)_r \\
\rightarrow \left((\Delta[m] \times \Theta[q](c_1, \ldots, c_q)) \times \Delta[p] \right)_r.
\end{aligned} \]

\normalsize

As arose in \cite[\S 4]{thesis}, some of these maps are not still monomorphisms after applying the reduction functor.  It suffices to take all maps as above where $m=q=p=0$, and where $m,q \geq 0$ and $p \geq 1$.  All other maps where $p=0$ either result in isomorphisms (which are unnecessary to include) or maps which are not isomorphisms.  For example, when $p=q=0$ and $m=1$, we obtain $\Delta[0] \amalg \Delta[0] \rightarrow \Delta[0]$ after reduction, which is not a monomorphism.  We denote by $I_c$ the set of remaining maps, which will be a set of generating cofibrations for one of our model structures.

However, this reduction process does not work as well when we seek to find generating cofibrations for a model structure analogous to the projective model structure on $\Thetaspdelta$, in which the generating cofibrations are of the form
\[ A \times \Delta[p] \rightarrow B \times \Delta[p] \] where $p\geq 0$ and $A \rightarrow B$ is a generating cofibration in $\Theta_nSp$.  For some of the maps $A \rightarrow B$ (in particular when, using the description of such maps above, $m=1$ or $q=1$), reduction does not give the correct map.

Thus, we also need to consider another set, first to prove a technical lemma for our first model structure, and then to be a set of generating cofibrations for the second model structure.  For any object $A$ in $\Theta_nSp$ and $p\geq 0$, define the object $A_{[p]}$ to be the pushout of the diagram
\[ \xymatrix{A \times (\Delta[p])_0 \ar[r] \ar[d] & A \times \Delta[p] \ar[d] \\
(\Delta[p])_0 \ar[r] & A_{[p]}. } \] Define the set
\[ I_f= \{A_{[p]} \rightarrow B_{[p]} \mid p \geq 0, A \rightarrow B \text{ a generating cofibration in } \Theta_nSp \}. \]

Let $X$ be a $\Thetansp$-Segal precategory, and consider the map $X \rightarrow \cosk_0X$.  Denote by $X_p(v_0, \ldots, v_p)$ the fiber of the map
\[ X_p \rightarrow (\cosk_0X)_p = X_0^{p+1}. \]  Then, for any object $A$ or $B$ as given above (noting that these objects are small in $\Theta_nSp$), we get
\[ \begin{aligned}
\Hom(A_{[p]},X) & =\Hom(A \times \Delta[p] \amalg_{A \times \Delta[p]_0} \Delta[p]_0, X) \\
&= \Hom(A, X_p) \times_{\Hom(A, X_0^{p+1})} X_0^{p+1} \\
& = \coprod_{v_0, \ldots, v_p} \Hom(A, X_p(v_0, \ldots, v_p)).
\end{aligned} \]
(Notice that by our assumption that $X$ is a $\Thetansp$-Segal precategory, $X_0$ is a discrete object of $\Theta_nSp$ and therefore our abuse of terminology that it has ``elements" $v_0, \ldots v_p$ makes sense.)

We make use of the following facts about fibrations in $\Thetansp$. We give the proof in Section \ref{fibrationproof}.

\begin{prop} \label{fibrations}
Let $X, X', Y$, and $Y'$ be objects of $\Thetansp$.
\begin{enumerate}
\item If $X$ and $Y$ are both discrete, then any map $X \rightarrow Y$ is a fibration.

\item If $X \rightarrow Y$ and $X \rightarrow Y$ be fibrations, then $X \amalg X' \rightarrow Y \amalg Y'$ is a fibration.
\end{enumerate}
\end{prop}

The following lemma is the higher analogue of \cite[4.1]{thesis}.

\begin{lemma} \label{mapping}
Suppose that a map $f \colon X \rightarrow Y$ of Segal precategory objects has the right lifting property with respect to the maps in $I_f$.  Then the map $X_0 \rightarrow Y_0$ is surjective, and each map
\[ X_p(v_0, \ldots, v_p) \rightarrow Y_p(fv_0, \ldots fv_p) \] is an acyclic fibration in $\Theta_nSp$ for each $p \geq 1$ and $(v_0, \ldots, v_p) \in X_0^{p+1}$.
\end{lemma}

\begin{proof}
Using our description of the generating cofibrations of $\Thetansp$, when $m=q=0$, we get the map $\varnothing \rightarrow \Delta[0]$. [Where did I define these???] The fact that $X \rightarrow Y$ has the right lifting property with respect to $\varnothing_{[0]} \rightarrow \Delta[0]_{[0]}$ implies that $X_0 \rightarrow Y_0$ is surjective.

To prove the remaining part of the statement, we need to show that a dotted arrow lift exists in all diagrams of the form
\[ \xymatrix{A \ar[r] \ar[d] & X_p(v_0, \ldots, v_p) \ar[d] \\
B \ar[r] \ar@{-->}[ur] & Y_p(fv_0, \ldots, fv_p) } \] for all choices of $p \geq 1$ and $A \rightarrow B$.  By our hypothesis, we have the existence of dotted arrow lifts
\[ \xymatrix{A_{[p]} \ar[r] \ar[d] & X \ar[d] \\
B_{[p]} \ar[r] \ar@{-->}[ur] & Y. } \]  The existence of such a lift is equivalent to the surjectivity of the map $\Hom(B_{[p]},X) \rightarrow P$, where $P$ is the pullback in the diagram
\[ \xymatrix{\Hom(B_{[p]},X) \ar[r] & P \ar[r] \ar[d] & \Hom(A_{[p]},X) \ar[d] \\
& \Hom(B_{[p]},Y) \ar[r] & \Hom(A_{[p]},Y). } \]

But, as we just showed above, we get
\[ \Hom(B_{[p]},X)= \coprod_{v_0, \ldots, v_p} \Hom(B, X_p(v_0, \ldots, v_p)), \] and analogously for the other objects in the diagram.  Looking at each component for each $(v_0, \ldots v_p)$ separately, we can check that surjectivity of this map does indeed give us the lift that we require.
\end{proof}

\begin{lemma} \label{mappingic}
Suppose that $f \colon X \rightarrow Y$ is a map in $(\Thetansp)^{\Deltaop}_{disc}$ with the right lifting property with respect to the maps in $I_c$.  Then
\begin{enumerate}
\item the map $f_0 \colon X_0 \rightarrow Y_0$ is surjective, and

\item for every $m \geq 1$ and $(v_0, \ldots, v_m) \in X_0^{n+1}$, the map
\[ X_m(v_0, \ldots, v_m) \rightarrow Y_m(fv_0, \ldots, fv_m) \] is a weak equivalence in $\Thetansp$.
\end{enumerate}
\end{lemma}

\begin{proof}
Since $f$ has the right lifting property with respect to the maps in the set $I_c$, it has the right lifting property with respect to all cofibrations.  In particular, $f$ has the right lifting property with respect to the maps in the set $I_f$.  Therefore, the result follows by Lemma \ref{mapping}.
\end{proof}

In order to give a precise definition of our weak equivalences, we need to define a ``localization" functor $L$ on the category $\Theta_nSp^{\Deltaop}_{disc}$ such that, for any object $X$, $LX$ is a Segal space object which is also a Segal category object weakly equivalent to $X$ in $\mathcal L_S\Theta_nSp^{\Deltaop}$.

To begin, we consider one choice of generating acyclic cofibrations in $\mathcal L_S\Theta_nSp^{\Deltaop}$, namely, the set
\[ \{C \times \Delta[p] \cup D \times G(p) \rightarrow D \times \Delta[p] \} \] where $p \geq 0$ and $C \rightarrow D$ is a generating acyclic cofibration in $\Theta_nSp$.  Using these maps, we can use the small object argument to construct a localization functor.

However, the maps with $p=0$ are problematic because taking pushouts along them, as given by the small object argument, results in objects which are no longer Segal category objects.  Thus, we consider maps as above, but with the restriction that $p \geq 1$.  To show that the ``localization" functor that results from this smaller set of maps is sufficient, in that it still gives us a Segal space object, we can use an argument just like the one given in \cite[\S 5]{thesis}.

Now, we make the following definitions in $\Theta_nSp^{\Deltaop}_{disc}$.
\begin{itemize}
\item Weak equivalences are the maps $f \colon X \rightarrow Y$ such that the induced map $LX \rightarrow LY$ is a Dwyer-Kan equivalence of Segal space objects.  (We call such maps \emph{Dwyer-Kan equivalences}.)

\item Cofibrations are the monomorphisms.

\item Fibrations are the maps with the right lifting property with respect to the maps which are both cofibrations and weak equivalences.
\end{itemize}

\begin{lemma} \label{rlpic}
Suppose that $f \colon X \rightarrow Y$ is a map in $(\Thetansp)^{\Deltaop}_{disc}$ with the right lifting property with respect to the maps in $I_c$.  Then $f$ is a Dwyer-Kan equivalence.
\end{lemma}

\begin{proof}
Suppose that $f \colon X \rightarrow Y$ has the right lifting property with respect to the maps in $I_c$.  By Lemma \ref{mappingic}, $f_0 \colon X_0 \rightarrow Y_0$ is surjective and each map
\[ X_m(v_0, \ldots, v_m) \rightarrow Y_m(fv_0, \ldots, fv_m) \] is a weak equivalence in $\Thetansp$ for $m \geq 1$ and $(v_0, \ldots, v_m) \in X_0^{m+1}$.  To prove that $f$ is a Dwyer-Kan equivalence, it remains to show that, for any $x,y \in X_0$, $\map_{LX}(x,y) \rightarrow \map_{LY}(fx,fy)$ is a weak equivalence in $\Thetansp$.

First, we construct a factorization of $f$ as follows.  Define $\Phi Y$ to be the pullback in the diagram
\[ \xymatrix{\Phi Y \ar[r] \ar[d] & Y \ar[d] \\
\cosk_0(X_0) \ar[r] & \cosk_0(Y_0). } \]
Then $(\Phi Y)_0 = X_0$ and, for every $m \geq 1$ and $(v_0, \ldots, v_m) \in X_0^{m+1}$, there is an isomorphisms of mapping objects
\[ (\Phi Y)_0(v_0, \ldots, v_m) \cong Y_m(fv_0, \ldots, fv_m). \]
Then $X \rightarrow \Phi Y$ is a Reedy weak equivalence and hence a Dwyer-Kan equivalence.  Therefore, it remains to prove that $\Phi Y \rightarrow Y$ is a Dwyer-Kan equivalence, via an inductive argument on the skeleta of $Y$.

For any $p \geq 0$, consider the map $\Phi (\sk_p Y) \rightarrow \sk_p Y$.  If $p=0$, then $\Phi (\sk_0 Y)$ and $\sk_0 Y$ are actually $\Thetansp$-Segal objects which can be observed to be Dwyer-Kan equivalent.  Therefore, assume that the map $\Phi (\sk_{p-1} Y) \rightarrow \sk_{p-1} Y$ is a Dwyer-Kan equivalence and consider the map $\Phi (\sk_p Y) \rightarrow \sk_p Y$.

We know that $\sk_p Y$ is obtained from $\sk_{p-1} Y$ via iterations of pushouts along maps $A_{[m]} \rightarrow B_{[m]}$ for $A \rightarrow B$ a generating cofibration in $\Thetansp$.  Since we need a more precise formulation, we recall that generating cofibrations in $\Thetansp$ are of the form
\[ \partial \Delta[m] \times \Theta[q](c_1, \ldots c_q) \cup \Delta[m] \times \partial \Theta[q](c_1, \ldots, c_q) \rightarrow \Delta[m] \times \Theta[q](c_1, \ldots, c_q) \] for $m, q \geq 0$ and $c_1, \ldots, c_q$ objects of $\Theta_{n-1}$.  So, we have the pushout diagram
\[ \xymatrix{\Delta[m] \times \Theta[q](c_1, \ldots, c_q) \times \Delta[p]_0 \ar[d] \ar[r] & \Delta[m] \times \Theta[q](c_1, \ldots, c_q) \times \Delta[p] \ar[d] \\
\Delta[p]_0 \ar[r] & (\Delta[m] \times \Theta[q](c_1, \ldots, c_q))_{[p]}.} \]  Similarly, we obtain $(\partial \Delta[m] \times \Theta[q](c_1, \ldots c_q) \cup \Delta[m] \times \partial \Theta[q](c_1, \ldots, c_q))_{[p]}$.

For simplicity, assume that we require only one pushout to obtain $\sk_p Y$ from $\sk_{p-1} Y$; here we further simplify by considering the case where $m=q=0$, although the argument can be extended more generally.  For this case, we have the pushout diagram
\[ \xymatrix{\varnothing \ar[r] \ar[d] & \sk_{p-1} Y \ar[d] \\
\Delta[p] \ar[r] & \sk_p Y.} \]
Since we know by our inductive hypothesis that $\Phi(\sk_{p-1} Y) \rightarrow \sk_{p-1} Y$ is a Dwyer-Kan equivalence, it suffices to establish that $\Phi \Delta[p] \rightarrow \Delta[p]$ is a Dwyer-Kan equivalence.  In the setting where these are levelwise discrete simplicial spaces, this fact was established in \cite[\S 9]{thesis}.  The argument given there continues to hold in the present case, making use of the fact that the model structure for $\Thetansp$-Segal spaces is cartesian.
\end{proof}

\begin{theorem} \label{cversion}
There is a cofibrantly generated model category structure $\mathcal L \Theta_nSp^{\Deltaop}_{disc,c}$ on the category of $\Thetansp$-Segal precategories with the above weak equivalences, fibrations, and cofibrations.
\end{theorem}

\begin{proof}
We use \cite[4.1]{beke} to establish this model structure.  It is not too hard to show that condition (1) is satisfied with $W$ the class of weak equivalences as defined.  However, to prove the remaining two statements we need the set
\[ I_c=\{(A \times \Delta[p] \cup B \times \partial \Delta[p])_r \rightarrow (B \times \Delta[p])_r \} \] where $A \rightarrow B$ are the generating cofibrations in $\Theta_nSp$.

Condition (2) was established in Lemma \ref{rlpic}.

For condition (3), first notice that elements of $\cof(I_c)$ are monomorphisms.  Now suppose that $X \rightarrow Y$ is a weak equivalence which is in $\cof(I_c)$, and suppose
\[ \xymatrix{X \ar[r] \ar[d] & Z \ar[d] \\
Y \ar[r] & W} \] is a pushout diagram.  Then notice that in the diagram
\[ \xymatrix{\map_{LX}(x,y) \ar[r] \ar[d] & \map_{LZ}(x,y) \ar[d] \\
\map_{LY}(x,y) \ar[r] & \map_{LW}(x,y)} \] again has the left-hand vertical map a cofibration and weak equivalence in $\Thetansp$, and is again a pushout diagram.  Furthermore, using the definition of homotopy category in a $\Thetansp$-Segal category, it can be shown that the analogous diagram of homotopy categories is again a pushout diagram.  Therefore, weak equivalences which are in $\cof(I_c)$ are preserved by pushouts.  A similar argument using mapping objects and homotopy categories establishes that such maps are preserved by transfinite compositions.
\end{proof}

We now define another model structure with the same weak equivalences, but for which the cofibrations are given by transfinite compositions of pushouts along the maps of the generating set $I_f$, and the fibrations are then determined.

\begin{theorem}
There is a model structure $\mathcal L (\Theta_nSp)^{\Deltaop}_{disc, f}$ on the category of Segal precategory objects with weak equivalences the Dwyer-Kan equivalences and the cofibrations given by iterated pushouts along the maps of the set $I_f$.
\end{theorem}

\begin{proof}
As before, we show that the conditions of \cite[4.1]{beke} are satisfied.  Condition (1) continues to hold from the previous model structure.  A similar proof can be used to establish condition (2), using Lemma \ref{mapping} and a proof analogous to the one for Lemma \ref{rlpic}.  Condition (3) works as in the other model structure.
\end{proof}

\section{Quillen equivalences between Segal category objects and enriched categories}

We now establish Quillen equivalences between the models given in the previous sections.

\begin{prop}
The identity functor induces a Quillen equivalence
\[ \xymatrix@1{\mathcal L (\Theta_nSp)^{\Deltaop}_{disc,c} \ar@<.5ex>[r] & \mathcal L (\Theta_nSp)^{\Deltaop}_{disc,f}. \ar@<.5ex>[l]} \]
\end{prop}

\begin{proof}
The identity map from $\mathcal L \Theta_nSp^{\Deltaop}_{f, disc}$ to $\mathcal L \Theta_nSp^{\Deltaop}_{c, disc}$ preserves cofibrations and acyclic cofibrations, so we get a Quillen pair.  The fact that it is a Quillen equivalence follows then from the fact that weak equivalences are the same in both categories.
\end{proof}

\begin{prop} \label{qpair}
There is a Quillen pair
\[ \xymatrix@1{F \colon \mathcal L (\Theta_nSp)^{\Deltaop}_{f, disc} \ar@<.5ex>[r] & \Theta_nSp-Cat \colon R. \ar@<.5ex>[l]} \]
\end{prop}

To prove this proposition, we make use of the following definition.

\begin{definition}
Let $\mathcal D$ be a small category, $\mathcal C$ a simplicial category, and $\mathcal C^\mathcal D$ the category of
functors $\mathcal D \rightarrow \mathcal C$.  Let $S$ be a set of
morphisms in $\SSetsd$.  An object $Y$ of $\mathcal C^\mathcal D$ is
\emph{strictly} $S$-\emph{local} if for every morphism $f \colon A
\rightarrow B$ in $S$, the induced map on function complexes
\[ f^*: \Map (B,Y) \rightarrow \Map (A,Y) \]
is an isomorphism of simplicial sets. A map $g:C \rightarrow D$ in
$\mathcal C^\mathcal D$ is a \emph{strict} $S$-\emph{local equivalence} if for
every strictly $S$-local object $Y$ in $\mathcal C^\mathcal D$, the induced map
\[ g^*:\Map(D,Y) \rightarrow \Map(C,Y) \]
is an isomorphism of simplicial sets.
\end{definition}

Here, we consider functors $\Deltaop \rightarrow \Theta_nSp$ which are discrete at level zero.  Notice that a category enriched in $\Theta_nSp-Cat$ can be regarded as a strictly local object in this category when we localize with respect to the map $\varphi$ described in an earlier section.  Recall that a Segal category object is a (nonstrictly) local object when regarded as a Segal space object $\Deltaop \rightarrow \Theta_nSp$.  Thus, the enriched nerve functor can be regarded as an inclusion map
\[ R \colon \Theta_nSp-Cat \rightarrow \Theta_nSp^{\Deltaop}. \]

Although we are working in the subcategory of functors which are discrete at level zero, we can still
use the following lemma to obtain a left adjoint functor $F$ to
our inclusion map $R$, since the construction will always produce
a diagram with discrete set at level zero when applied to such a diagram.

\begin{lemma}
For any small category $\mathcal D$ and any model category $\mathcal M$, consider the category of all diagrams $X: \mathcal
D \rightarrow \mathcal M$ and the category of strictly local diagrams
with respect to the set of maps $S= \{f:A \rightarrow B\}$. The
forgetful functor from the category of strictly local diagrams to
the category of all diagrams has a left adjoint.
\end{lemma}

\begin{proof}
This lemma was proved in \cite[5.6]{multisort} in the case where $\mathcal M=\SSets$, but the proof continues to hold if we use a more general simplicial category.
\end{proof}

We define $F \colon \mathcal L(\Thetansp)^{\Deltaop}_{disc, f} \rightarrow \Theta_nSp-Cat$ to be
this left adjoint to the inclusion map of strictly local diagrams.

\begin{proof}[Proof of Proposition \ref{qpair}]
To prove this proposition, we modify the approach given in the proof of the analogous result when $n=1$ \cite[8.3]{thesis}.  We first show that $F$ preserves cofibrations.  Since $F$ is a left adjoint functor, we know that it preserves colimits, so it suffices to show that $F$ takes the maps in the set $I_f$ to cofibrations in $\Theta_nSp-Cat$.

Let $\ast$ denote the terminal object in $(\Thetansp)^{\Deltaop}$.  Since cofibrations are inclusions in $\Theta_nSp$, the map $\varnothing \rightarrow \ast$ is a cofibration, and $\varnothing_{[0]} \rightarrow \ast_{[0]}$ is already local; in fact it corresponds to the generating cofibration $\varnothing \rightarrow \{x\}$ in $\Theta_nSp-Cat$.

%Using our standard description of the generating cofibrations $A \rightarrow B$ in $\Theta_nSp$, when $A=\varnothing$ and $B=\Delta[0]$, when we consider the map $A_{[0]} \rightarrow B_{[0]}$ in $I_f$ we get objects that are already local and in fact correspond exactly to the map $\varnothing \rightarrow \{x\}$ in the generating cofibrations for $\Theta_nSp-Cat$.

For any generating cofibration $A \rightarrow B$, localizing the map $A_{[1]} \rightarrow B_{[1]}$ results in the generating cofibration $UA \rightarrow UB$ of $\Theta_nSp-Cat$.  Localizing any other map of $I_f$ results in a map in $\Theta_nSp-Cat$ which is a colimit of maps of this form, and therefore $F$ preserves cofibrations.

To show that $F$ preserve acyclic cofibrations, we use the Quillen equivalence in the fixed-object set situation; the argument given in \cite[8.3]{thesis} still holds in this more general setting.
\end{proof}

To prove that this Quillen pair is a Quillen equivalence, we use the following theorem, which is the analogue of \cite[8.5]{thesis}.

\begin{lemma} \label{cofibrant}
For every cofibrant object $X$ in $\mathcal L (\Theta_nSp)^{\Deltaop}_{disc,f}$, the map $X \rightarrow FX$ is a Dwyer-Kan equivalence.
\end{lemma}

\begin{proof}
Consider an object in $\mathcal L \Theta_nSp^{\Deltaop}_{disc, f}$ of the form $\coprod_iB_{[p_i]}$, where $B$ is the target of a generating cofibration of $\Theta_nSp$, and let $Y$ be a fibrant object of $\mathcal L \Theta_nSp^{\Deltaop}_{disc, f}$.  Then notice that $(\Delta [p] \times B)_k = \Delta[p]_k \times B$ since $B$ is regarded as a constant simplicial diagram.  Then
\[ \begin{aligned}
\Map(\Delta[m], \Delta[p] \times B) & \cong \Map(\Delta[m], \Delta[p]) \times \Map(\Delta[m], B) \\
& \cong \Map(G(m), \Delta[p]) \times \Map(G(m), B) \\
& \cong \Map(G(m), \Delta[p], B)
\end{aligned} \] so $\Delta[p] \times B$ is strictly local.  By its construction, it follows that $\coprod_i B_{[p_i]}$ is also strictly local.  In particular, the map
\[ \coprod_i B[p_i] \rightarrow F\left( \coprod_i B[p_i] \right) \] is a Dwyer-Kan equivalence.

Now, suppose that $X$ is any cofibrant object.  Then it can be written as a colimit of objects of the above form, and we can assume that it can be written as
\[ X \simeq \colim_{\Deltaop}X_j \] where $X_j=\coprod_I B[p_i]$.  Then, using arguments about mapping spaces and strictly local objects as in \cite[8.5]{thesis}, we can show that
\[ \Map(X,Y) \simeq \Map(FX,Y) \] for any strictly local fibrant object $Y$, completing the proof.
\end{proof}

\begin{theorem}
The Quillen pair
\[ \xymatrix@1{F \colon \mathcal L(\Thetansp)^{\Deltaop}_{f, disc} \ar@<.5ex>[r] & \Theta_nSp-Cat \colon R. \ar@<.5ex>[l]} \]
is a Quillen equivalence.
\end{theorem}

\begin{proof}
To prove this result, we can use Lemma \ref{cofibrant} to prove that $F$ reflects weak equivalences between cofibrant objects.  Then, we show that for any fibrant $\Theta_nSp$-category, the map $F((RY)^c) \rightarrow Y$ is a Dwyer-Kan equivalence, where $(RY)^c$ denotes a cofibrant replacement of $RY$.  The proof follows just as in the $n=1$ case \cite[8.6]{thesis}.
\end{proof}

\section{Fibrations in $\Thetansp$} \label{fibrationproof}

In this section we give the proof of Proposition \ref{fibrations}, establishing properties of fibrations in $\Thetansp$.

We begin with the case where $n=1$, so that $\Thetansp$ is just $\css$, the model structure for complete Segal spaces.

\begin{prop} \label{cssfibs}
The statement of Proposition \ref{fibrations} holds when $n=1$.
\end{prop}

\begin{proof}
Recall that the generating acyclic cofibrations in $\css$ are of the form
\[ V[m,k] \times \Delta[p]^t \cup \Delta [m] \times G(p)^t \rightarrow \Delta[m] \times \Delta[p]^t \] or
\[ V[m,k] \times E^t \cup \Delta[m] \times \Delta[0]^t \rightarrow \Delta[m] \times E^t \] where $m \geq 1$, $0 \leq k \leq m$, $p \geq 0$, and $E$ denotes the nerve of the category with two objects and a single isomorphism between them.

Suppose that $X$ and $Y$ are discrete simplicial spaces.  To show that any map $X \rightarrow Y$ is a fibration, it suffices to prove that that it has the right lifting property with respect to these two kinds of generating acyclic cofibrations, which is equivalent to the existence of dotted-arrow lifts in the diagrams of simplicial sets
\[ \xymatrix{V[m,k] \ar[r] \ar[d] & X_n \ar[d] & \\
\Delta[m] \ar[r] \ar@{-->}[ur] & P \ar[r] \ar[d] & X_1 \times_{X_0} \cdots \times_{X_0} X_1 \ar[d] \\
& Y_n \ar[r] & Y_1 \times_{Y_0} \cdots \times_{Y_0} Y_1} \] and
\[ \xymatrix{V[m,k] \ar[r] \ar[d] & \Map(E^t,X) \ar[d] & \\
\Delta[m] \ar[r] \ar@{-->}[ur] & Q \ar[r] \ar[d] & X_0 \ar[d] \\
& \Map(E^t, Y) \ar[r] & Y_0 } \] where $P$ and $Q$ denote the pullbacks of their respective lower square diagrams.
In the first diagram, since $X$ and $Y$ are discrete, $X_0=X_1=X_n$ and $Y_0=Y_1=Y_n$ for all $n \geq 2$, so $P=X_n$ and the right-hand vertical map in the upper square is an isomorphism.  Therefore, the necessary lift exists.  Similarly, in the second diagram, we can again use the fact that $X$ and $Y$ are discrete to show that $\Map(E^t,X)=X_0$ and $\Map(E^t,Y)=Y_0$, from which it follows that $Q=X_0$ and the right-hand vertical map in the upper diagram is an isomorphisms, implying the existence of the desired lift.  Therefore, we have established that (1) holds in $\css$.

For (2), Suppose that $X \rightarrow Y$ and $X' \rightarrow Y'$ have the right lifting property with respect to the two kinds of generating acyclic cofibrations.  For the first kind, we need to find a dotted-arrow lift in any diagram of the form
\[ \xymatrix{V[m,k] \ar[r] \ar[d] & (X \amalg X')_n \ar[d] & \\
\Delta[m] \ar[r] \ar@{-->}[ur] & P \ar[r] \ar[d] & (X \amalg X')_1 \times_{(X \amalg X')_0} \cdots \times_{(X \amalg X')_0} (X \amalg X')_1 \ar[d] \\
& (Y \amalg Y')_n \ar[r] & (Y \amalg Y')_1 \times_{(Y \amalg Y')_0} \cdots \times_{(Y \amalg Y')_0} (Y \amalg Y')_1. } \]
However, since all maps in sight are given by coproducts of maps, we can rewrite the right-hand vertical map in the lower diagram as
\[ (X_1 \times_{X_0} \cdots \times_{X_0} X_1) \amalg (X'_1 \times_{X'_0} \cdots \times_{X'_0} X'_1) \rightarrow (Y_1 \times_{Y_0} \cdots \times_{Y_0} Y_1) \amalg (Y'_1 \times_{Y'_0} \cdots \times_{Y'_0} Y'_1). \]  Since $\Delta[m]$ is connected, finding a lift reduces to finding a lift on one of the components, which holds since we have assumed that each component map $X \rightarrow X'$ or $X' \rightarrow Y'$ is a fibration.  A similar argument can be used to establish the right lifting property with respect to the second type of acyclic cofibration.
\end{proof}

The proof of Proposition \ref{fibrations} can then be established via the following inductive result.

\begin{prop}
If conditions (1) and (2) from Proposition \ref{fibrations} hold for $\Theta_{n-1}Sp$, $n \geq 2$, then they hold for $\Thetansp$.
\end{prop}

\begin{proof}
The generating acyclic cofibrations of $\Thetansp$ are of three kinds:
\[ V[m,k] \times \Theta_p(c_1, \ldots, c_p) \cup \Delta[m] \times G(p)(c_1, \ldots, c_p) \rightarrow \Delta[m] \times \Theta_p(c_1, \ldots, c_p) \]
for $m \geq 1$, $0 \leq k \leq m$, $p \geq 0$, and $c_1, \ldots, c_p$ objects of $\Theta_{n-1}$,
\[ V[m,k] \times T_\# \Delta [0] \cup \Delta[m] \times T_\# E \rightarrow \Delta[m] \times T_\# \Delta[0], \] for $m,k$ as before, and
\[ V[m,k] \times V[1](B) \cup \Delta[m] \times V[1](A) \rightarrow \Delta[m] \times V[1](B) \] where $A \rightarrow B$ is a map in $\mathcal T_{n-1}$, the set of generating cofibrations for $\Theta_{n-1}Sp$.

Let us first consider the case where $X \rightarrow Y$ is a map between discrete objects.  Showing that this map has the right lifting property with respect to the first two kinds of generating acyclic cofibrations is analogous to the proof of Proposition \ref{cssfibs}.  For the third kind, we need to show the existence of a dotted-arrow lift in any diagram of the form
\[ \xymatrix{V[m,k] \ar[r] \ar[d] & \Map(V[1](B), X) \ar[d] & \\
\Delta[m] \ar[r] \ar@{-->}[ur] & P \ar[r] \ar[d] & \Map(V[1](A), X) \ar[d] \\
& \Map(V[1](B), Y) \ar[r] & \Map(V[1](A), Y)} \] where $P$ denotes the pushout of the lower square.

Now, recall from \cite{rezktheta} that we can define the mapping object $M_X(x_0, x_1)(c_1)$ to be the object of $\Theta_{n-1}Sp$ defined as the pullback in the diagram
\[ \xymatrix{M_X(x_0, x_1)(c_1) \ar[r] \ar[d] & X[1](c_1) \ar[d] \\
{(x_0, x_1)} \ar[r] & X[0] \times X[0].} \]  Furthermore, we get
\[ \Map(V[1](B), X) = \coprod_{x_0, x_1} \Map(B, M_X(x_0, x_1)) \] and analogously for other objects in the above diagram.  Since we have reduced the problem to the world of $\Theta_{n-1}Sp$, our inductive hypothesis shows that the necessary lift exists.  Hence, condition (1) holds.

The same kind of argument, and again using the ideas of the proof of Proposition \ref{cssfibs}, we can verify that condition (2) holds as well.
\end{proof}

\end{document}